\documentclass[leqno]{etna} 
\usepackage[utf8]{inputenc}
\usepackage{amsmath}
\usepackage{amsfonts}
\usepackage{graphicx}
\usepackage{algorithmic}
\usepackage{tikz}
\usepackage{smartdiagram}
\usepackage{float}
\floatstyle{ruled}
\newfloat{algorithm}{tbp}{loa}
\floatname{algorithm}{Algorithm}
\usepackage{algorithmic}
\newcommand{\ru}{\mathbf{u}}		%	roundoff unit

\newcommand{\norm}[1]{\Vert #1 \Vert}

\setbibdata{1}{23}{xx}{20xx} %{first page}{last page}{volume}{year}

\hypersetup{
  pdftitle={Exact Lanczos algorithm},
  pdfauthor={Dorota \v{S}imonov\'a and Petr Tich{\'y}}
  pdfkeywords={Lanczos algorithm, exact computations, finite precision arithmetic, rounding errors}
}

\title{When does the lanczos algorithm compute exactly?\thanks{The work of P. Tich{\'y} was supported by the Grant Agency of the Czech Republic under the grant  no. 20-01074S.
}}

\author{Dorota \v{S}imonov\'a\footnotemark[2] \and Petr Tich{\'y}\footnotemark[2]}

\shorttitle{Exact Lanczos algorithm}
\shortauthor{D. \v{S}imonov\'a and P. Tich{\'y}}

\begin{document}
\maketitle
\renewcommand{\thefootnote}{\fnsymbol{footnote}}
\footnotetext[2]{Faculty of Mathematics and Physics, Charles University, Prague, Czech Republic. \\
e-mail: {\tt simonova@karlin.mff.cuni.cz}, {\tt petr.tichy@mff.cuni.cz}.}

\begin{abstract}
In theory, the Lanczos algorithm generates an orthogonal basis of
the corresponding Krylov subspace. However, in finite precision arithmetic,
the orthogonality and linear independence of the computed Lanczos
vectors is usually lost quickly. In this paper we study a class of
matrices and starting vectors having a special nonzero structure that
guarantees exact computations of the Lanczos algorithm whenever floating point arithmetic satisfying the IEEE 754 standard is used. Analogous
results are formulated also for a variant of the conjugate gradient
method that produces then almost exact results. The results are extended
to the Arnoldi algorithm, the nonsymmetric Lanczos algorithm, the Golub-Kahan bidiagonalization, the block-Lanczos algorithm and their counterparts for solving linear systems.
\end{abstract}



\section{Introduction}
Let a real and symmetric matrix $A$ and a starting vector $v$
be given. The Lanczos algorithm is a frequently used
algorithm for computing an orthogonal basis of the corresponding Krylov
subspace. At the same time, it can be
seen as a method for approximating a few eigenvalues (and eventually
eigenvectors) of $A$, using the underlying Rayleigh-Ritz procedure;
see, e.g., \cite{B:Pa1980}.

Since the introduction of the algorithm in 1950
by Lanczos \cite{La1950} 
it has been known
that the orthogonality of the
computed basis vectors need not be preserved due to rounding errors.
As a consequence, an eigenvalue of $A$ can be approximated by several
eigenvalues of the Jacobi matrix produced by the 
Lanczos algorithm in finite precision arithmetic. 

The numerical behavior of the Lanczos algorithm was analyzed by 
Paige \cite{Pa1976,Pa1980a}. Paige showed that the effects
of rounding errors on the Lanczos algorithm can be described mathematically.
Based on these results, Greenbaum \cite{Gr1989} proved that
the results of finite precision computations can be interpreted
as the results of the exact Lanczos algorithm applied to a larger
problem with a matrix having many eigenvalues distributed throughout
tiny intervals around the eigenvalues of $A$. In other words, Greenbaum
found and constructed a mathematical model of the finite precision Lanczos
computations. In particular, Greenbaum's model matrix is a Jacobi
matrix, and the starting vector is a multiple of the first column
$e_{1}$ of the identity matrix. Results of Paige and Greenbaum stimulated
further development in the analysis of the numerical behavior of
the Lanczos and the conjugate gradient (CG) algorithms; see, e.g., \cite{St1991,GrSt1992,W2005,Wu2006}.
% StGr1992,
For a comprehensive summary and a detailed explanation; see \cite{MeSt2006}.

In this paper we prove and extend an interesting observation made
by Marie Kub\'{\i}nov\'a in her PhD thesis \cite[p. 77]{T:Ku2018}: {\em If the
Lanczos algorithm is applied to a Jacobi matrix and a multiple of
$e_{1}$, then no rounding errors appear.} In other words, the finite
precision Lanczos algorithm computes exactly. Note that we also formulate
analogous statement for a variant of the CG algorithm that provides,
for the above mentioned input data, almost exact results (within
the relative accuracy given by machine precision). The obtained  results have
several consequences discussed in detail in Section~\ref{sec:cons}, that could
be useful in further analysis of the behavior of the Lanczos and CG
algorithms. For example, they allow to investigate experimentally the theoretical behavior
of the Lanczos algorithm for potentially very large systems by forming
a tridiagonal matrix with the desired properties, and then running
the Lanczos algorithm with the starting vector $e_{1}$ without reorthogonalization. 

The paper is organized as follows. In Sections \ref{sec:lan} and \ref{sec:Exact} we recall the standard version of the Lanczos algorithm, and summarize operations
and transformations that are performed exactly in floating point arithmetic.
Section~\ref{sec:flan} investigates a nonzero structure of
the input data, that ensures the exact computations of the Lanczos algorithm
in floating point arithmetic satisfying the IEEE 754 standard. In
Section~\ref{sec:cg} we formulate analogous results for a variant of the CG method.
Section~\ref{sec:other} shows that the results of Sections~\ref{sec:flan} and \ref{sec:cg} can be generalized to other algorithms like the Arnoldi algorithm, the nonsymmetric Lanczos algorithm, the Golub-Kahan bidiagonalization and the block-Lanczos algorithm and their counterparts for solving
linear systems. Finally, in Section~\ref{sec:cons} we discuss consequences and a possible use of the obtained results.

\section{Lanczos algorithm}
\label{sec:lan}
Given a starting vector $v\in\mathbb{R}^{n}$ and a symmetric matrix
$A\in\mathbb{R}^{n\times n}$, one can consider a sequence of nested
subspaces 
\[
\mathcal{K}_{k}(A,v)=\mathrm{span}\{v,Av,\dots,A^{k-1}v\}%\mathrm{span}\{v_{1},\dots,v_{j}\},\quadj=1,\dots,k+1.
\]
called the Krylov subspaces. The dimension of these subspaces is increasing
up to an index $d=d(A,v)$ called the \emph{degree of $v$ with respect
to $A$}, for which the maximal dimension is attained, and $\mathcal{K}_{d}(A,v)$
is invariant under multiplication with $A$. Having an index $k<d$,
the Lanczos algorithm (Algorithm~\ref{alg:lanczos}) 
\begin{algorithm}[th]
\caption{Lanczos algorithm}
\label{alg:lanczos}

\begin{algorithmic}[1]

\STATE \textbf{input} $A$, $v$ 

\STATE $\beta_{1}=\|v\|$, $v_{0}=0$ 

\STATE $v_{1}=v/\beta_{1}$ 

\FOR{$i=1,\dots,k$} 

\STATE $w=Av_{i}-\beta_{i}v_{i-1}$ 

\STATE $\alpha_{i}=w^{T}v_{i}$ 

\STATE $z=w-\alpha_{i}v_{i}$ 

\STATE $\beta_{i+1}=\|z\|$ 

\STATE\textbf{if} $\beta_{i+1}=0$ \textbf{then} \textbf{stop} 

\STATE $v_{i+1}=z/\beta_{i+1}$ 

\ENDFOR 

\end{algorithmic} 
\end{algorithm}
constructs an orthonormal basis $v_{1},\dots,v_{k+1}$ of the Krylov
subspace $\mathcal{K}_{k+1}(A,v)$. The Lanczos vectors $v_{j}$ satisfy
the three-term recurrence 
\begin{align}
\beta_{i+1}v_{i+1}=Av_{i}-\alpha_{i}v_{i}-\beta_{i}v_{i-1},\qquad i=1,\dots,k\label{threeterm}
\end{align}
or, written in the matrix form, 
\[
AV_{k}=V_{k}T_{k}+\beta_{k+1}v_{k+1}e_{k}^{T}
\]
where $V_{k}=[v_{1},\dots,v_{k}]$, the vector $e_{k}$ denotes the
$k$th column of the identity matrix of an appropriate size (here
of the size $k$), and $T_{k}$ is the $k$ by $k$ symmetric tridiagonal
matrix of the Lanczos coefficients, 
\begin{equation}
T_{k}=\left[\begin{array}{cccc}
\alpha_{1} & \beta_{2}\\
\beta_{2} & \ddots & \ddots\\
 & \ddots & \ddots & \beta_{k}\\
 &  & \beta_{k} & \alpha_{k}
\end{array}\right].\label{matrixT}
\end{equation}
Since the coefficients $\beta_{j}$ are positive, $T_{k}$ is a Jacobi
matrix. The Lanczos algorithm works for any symmetric matrix, but
if $A$ is positive definite, then $T_{k}$ is positive definite as
well.

During computations in floating point arithmetic, rounding errors
may have a significant influence on the computed results. In particular,
the orthogonality among the Lanczos vectors is usually lost very quickly.
In this paper we are interested in happy cases when this situation
does not happen. In more detail, assuming that $d=n$ and considering 
the standard model of floating point arithmetic
that satisfies the IEEE 754 standard, we look
for a nonzero pattern of $A$ and $v$ such that no rounding errors
appear during the computation of the Lanczos algorithm. The classical
examples of arithmetics satisfying the IEEE 754 standard are the double
precision (binary64), single precision (binary32), or half precision
(binary16).

\section{\label{sec:Exact}Exact computations in floating point arithmetic}
Let $\mathbb{F}$ denote the set of floating point numbers 
%If $\alpha\in\mathbb{R}$,
%then $\mathrm{fl}(\alpha)$ is its floating point representation and
%we assume that
%\begin{equation}\label{eqn:fl1}
%\mathrm{fl}(\alpha)=\alpha(1+\delta),\qquad|\delta|\leq\ru,
%\end{equation}
%where $\ru$ is the \emph{unit roundoff}. 
and let ``$\circ$'' is one of the basic operations (addition, subtraction, multiplication, division, square root).
Suppose that $\alpha$ and $\beta$
are floating point numbers and 
that $\alpha\circ\beta$ is within the exponent range (otherwise we get {\em overflow} or {\em underflow}). Denote 
the floating point result by 
$\mathrm{fl}(\alpha\circ\beta)$.
Then, considering 
the standard model of floating point arithmetic,
it holds that
\[
\mathrm{fl}(\alpha\circ\beta)=(\alpha\circ\beta)(1+\delta),\qquad|\delta|\leq\ru,
\]
where $\ru$ is the \emph{unit roundoff}. 
Obviously, if $\alpha\in\mathbb{F}$,
then 
\[
\mathrm{fl}(1*\alpha)=\alpha,\quad\mathrm{fl}(-\alpha)=-\alpha,\quad\mathrm{fl}(0*\alpha)=0,\quad\mathrm{fl}(\alpha-\alpha)=0,\quad\mathrm{fl}(\alpha/\alpha)=1.
\]
It is easy to see that if $P\in\mathbb{F}^{n\times n}$ is a permutation
matrix, $v\in\mathbb{F}^{n}$, $A\in\mathbb{F}^{n\times n}$, then
\[
\mathrm{fl}(P^{T}P)=I,\quad\mathrm{fl}(Pv)=Pv,\quad\mathrm{fl}(PA)=PA,\quad\mathrm{fl}(AP)=AP.
\]
%Let $\mathtt{realmax}$ (to be consistent with the Matlab notation)
%be the largest finite floating point number, and $\mathtt{realmin}$ 
%the smallest positive normalized floating point number. 
In the following lemma we show
that if $\alpha\in\mathbb{F}$ and if $\alpha^{2}$ is within the exponent range, then the square root of the second power
of $\alpha$ is computed exactly; see also \cite[Question 1.17]{B:De1997}.\smallskip
\begin{lemma}
\label{lemmatko} Consider the standard model of floating point arithmetic.
Let $\alpha\in\mathbb{F}$ be a floating point number such that $\alpha^{2}$ is within the exponent range.
%$\mathtt{realmin}\leq\alpha^{2}\leq\mathtt{realmax}$.
Then it holds that 
\[
|\alpha|=\mathrm{fl}\left(\sqrt{\mathrm{fl}\left(\alpha^{2}\right)}\right).
\]
\end{lemma}
\begin{proof}
Assume without loss of generality that $\alpha\geq0$, otherwise we
replace $\alpha$ by $|\alpha|$ in the text below. For $\beta\equiv\mathrm{fl}(\alpha^{2})$
it holds that 
\[
\beta=\alpha^{2}(1+\delta),\qquad|\delta|\leq\ru,
\]
and the exact square root of $\beta\in\mathbb{F}$ is given by 
\begin{equation}
\sqrt{\beta}=\alpha\sqrt{1+\delta}=\alpha\left(1+\frac{\delta}{2}-\frac{\delta^{2}}{8}+\frac{\delta^{3}}{16}+O(\delta^{4})\right)=\alpha\left(1+\frac{\delta}{2}+O(\ru^{2})\right),\label{eqn:sqrtb}
\end{equation}
where we have used the Taylor expansion of $\sqrt{1+\delta}$.

The IEEE 754 standard of floating point arithmetic guarantees
that $\mathrm{fl}(\sqrt{\beta})$ is the nearest floating point number
to the exact value of $\sqrt{\beta}$. Since $\alpha$ is a floating
point number, the two nearest floating point numbers to $\alpha$
are given by $\alpha(1\pm2\ru)$, where $2\ru$ is the machine epsilon.
In other words, 
\begin{equation}
\alpha(1-2\ru),\ \alpha,\ \alpha(1+2\ru),\label{eqn:3a}
\end{equation}
are three consecutive floating point numbers. Comparing \eqref{eqn:sqrtb}
and \eqref{eqn:3a}, the nearest floating-point number to $\sqrt{\beta}$
is $\alpha$. 
\end{proof}
\smallskip

Considering a vector 
\begin{equation}
z=\alpha e_{j},\quad\alpha\in\mathbb{F},
%\quad\mathtt{realmin}\leq \alpha^{2}\leq\mathtt{realmax},
\label{eq:special}
\end{equation}
such that $\alpha^2$ is within the exponent range, then
the previous lemma shows that the Euclidean norm of $z$ is in the standard
model of floating point arithmetic computed exactly. On the other
hand, if $z$ is not a multiple of $e_{j}$, then, in general, one
can expect that rounding errors occur. In other words, the only
structure of $z$ that guarantees that no rounding errors occur during the computation of its Euclidean norm is the structure \eqref{eq:special}.

\section{Lanczos algorithm in floating point arithmetic}
\label{sec:flan}

On line 8 of Algorithm~\ref{alg:lanczos}, the Euclidean norm of the vector
$z$ is computed. To guarantee that the Lanczos algorithm
computes exactly for any matrix and any starting vector having a given
structure, the Lanczos vectors must
necessarily be equal to the columns of the identity matrix (up to
the sign); see \eqref{eq:special} and the discussion herein. In particular, since the normalized starting vector is
the first Lanczos vector $v_{1}$, it must hold that $v_{1}=\pm e_{j}$
for some $j=1,\dots,n$. To simplify the notation, we define the \textit{signed permutation matrix} as the permutation matrix with the entries $\pm 1$ instead of $1$. In the following lemma we investigate the parametrization of all matrices $A$ and vectors $v$ with $d=n$
such that the exact Algorithm~\ref{alg:lanczos} produces Lanczos vectors
having just one nonzero entry.\smallskip

\begin{lemma}\label{lem:equiv} 
Assuming exact arithmetic, Algorithm~\ref{alg:lanczos}
applied to a symmetric $A\in\mathbb{R}^{n\times n}$ and 
$v\in\mathbb{R}^{n}$ such that $d=n$ produces Lanczos vectors equal
to plus or minus columns of the identity matrix if and only if 
\[
A=PTP^{T},\qquad v=\widetilde{\beta}_{1}Pe_{1},
\]
with $P\in\mathbb{R}^{n\times n}$ being a signed permutation matrix, and $T\in\mathbb{R}^{n\times n}$ being a tridiagonal matrix of the form 
\[
T=\left[\begin{array}{cccc}
\widetilde{\alpha}_{1} & \widetilde{\beta}_{2}\\
\widetilde{\beta}_{2} & \ddots & \ddots\\
 & \ddots & \ddots & \widetilde{\beta}_{n}\\
 &  & \widetilde{\beta}_{n} & \widetilde{\alpha}_{n}
\end{array}\right],
\]
where \textup{$\widetilde{\beta}_{j}>0$, $j=1,\dots,n$.} Moreover, the tridiagonal matrix $T_n$ resulting from Algorithm~\ref{alg:lanczos} is equal to $T$.
\end{lemma}\smallskip

\begin{proof}
Suppose first that the Lanczos vectors are equal to plus or minus columns
of the identity matrix and that $d=n$, i.e., there is a signed permutation
matrix $P$ such that $V_{n}=P.$ Since $d=n$, we
obtain in the last iteration of the Lanczos algorithm $AV_{n}=V_{n}T_{n}$
so that 
\[
A=V_{n}T_{n}V_{n}^{T}=PTP^{T},
\]
where we set $T = T_n$. Moreover,
$v_{1}=Pe_{1}$, and, therefore, the starting vector $v$ has
to have the form $v=\widetilde{\beta}_{1}Pe_{1}$ for some $\widetilde{\beta}_{1}>0.$

On the other hand, suppose that $A=PTP^{T}$ and $v=\widetilde{\beta}_{1}Pe_{1}$
for some signed permutation matrix~$P$ and $\widetilde{\beta}_{1}>0$. Applying the Lanczos algorithm to $A$ and $v$, we get
\begin{equation}\label{eq:aux1}
AV_n=V_{n}T_{n}    .
\end{equation}
The choice of $v$ ensures that the first column $v_1$ of $V_n$ is equal to the first column $p_1$ of $P$. Moreover, from the assumption on the structure of $A$ it follows 
\begin{equation}\label{eq:aux2}
AP=PT.
\end{equation}
Comparing \eqref{eq:aux1} and \eqref{eq:aux2} and using the fact that 
$p_1=v_1$, $P$ is orthogonal, and $T$ is Jacobi with positive off-diagonal entries, we obtain 
$T_{n} = T$ and $V_{n} = P$.  
\end{proof}\smallskip

In the following theorem we show that the structure of $A$ and
$v$ introduced in Lemma~\ref{lem:equiv} is sufficient for the Lanczos
algorithm to compute exactly in the standard floating point arithmetic. \smallskip

\begin{theorem}\label{thm:main}
Consider the standard model of floating point arithmetic.
Let 
\begin{equation}
A=PTP^{T},\qquad v=\widetilde{\beta}_{1}Pe_{1},\label{eq:structure}
\end{equation}
where $P\in\mathbb{F}^{n\times n}$ is a signed permutation matrix, and $T\in\mathbb{F}^{n\times n}$
is tridiagonal of the form 
\[
T=\left[\begin{array}{cccc}
\widetilde{\alpha}_{1} & \widetilde{\beta}_{2}\\
\widetilde{\beta}_{2} & \ddots & \ddots\\
 & \ddots & \ddots & \widetilde{\beta}_{n}\\
 &  & \widetilde{\beta}_{n} & \widetilde{\alpha}_{n}
\end{array}\right],
\]
with $\widetilde{\beta}_{j}>0$
and $\widetilde{\beta}_{j}^{2}$
within the exponent range.
%$\mathtt{realmin}\leq \widetilde{\beta}_{j}^{2}\leq\mathtt{realmax}$, $j=1,\dots,n$.
Then Algorithm~\ref{alg:lanczos}
applied to $A$ and $v$ computes exactly, i.e., no rounding errors appear
%are produced
during the computations. As a consequence, it holds that $T_n = T$.
\end{theorem} \smallskip
\begin{proof}
The proof is by induction.
Let us denote by bar the results of the computations in floating point arithmetic.
We start on lines 2 and 3 of Algorithm~\ref{alg:lanczos}.
It is easy to check that $\bar{v} = \mathrm{fl}(\widetilde{\beta}_{1}Pe_{1}) = v$, 
$\bar{v}_0 = 0 = v_0$, 
$\bar{\beta}_1 = \mathrm{fl}(\|v\|) = \widetilde{\beta}_{1} = \beta_1$,
and $\bar{v}_1 = \mathrm{fl}(v/\beta_1) = v_1$
are computed exactly.

Define the vector $e_0 = 0$ and assume that for $1 \leq i \leq n-1$ the vectors 
$
v_{j}=Pe_{j}
$, $j=0,\dots,i$,
the coefficients $\alpha_{j} = \widetilde{\alpha}_{j}$,
$j=1,\dots,i-1$,
and 
$\beta_{j}=\widetilde{\beta}_{j}$,
$j=1,\dots,i$
are computed exactly. 
Using results of Section~\ref{sec:Exact},
the induction hypothesis, 
and 
observing that $\bar{w} = \mathrm{fl}(\mathrm{fl}(Av_{i})-\mathrm{fl}(\beta_{i}v_{i-1}))=P\mathrm{fl}(Te_{i}-\beta_{i}e_{i-1})$, 
we obtain on line 5
\begin{eqnarray*}
\bar{w}=
P\mathrm{fl}\left(
\widetilde{\beta}_i e_{i-1} + 
\widetilde{\alpha}_{i} e_{i}
+
\widetilde{\beta}_{i+1} e_{i+1}
- \beta_{i}e_{i-1}\right)
=
P\left(
\widetilde{\alpha}_{i} e_{i}
+
\widetilde{\beta}_{i+1} e_{i+1}
\right)
=w.
\end{eqnarray*}
Further, on line 6 we get
\[
\bar{\alpha}_{i}=\mathrm{fl}(w^{T}v_{i})=\widetilde{\alpha}_{i}=\alpha_{i},
\]
and, using $\mathrm{fl}(\alpha_{i}v_{i})=\alpha_{i}Pe_{i}$, on line 7
\[
\bar{z}=\mathrm{fl}(w-\mathrm{fl}(\alpha_{i}v_{i}))
=
P\mathrm{fl}\left(
\alpha_{i} e_{i}
+
\widetilde{\beta}_{i+1} e_{i+1}
- \alpha_{i}e_{i}
\right)
=
\widetilde{\beta}_{i+1}Pe_{i+1}=z.
\]
Hence, $z=\bar{z}$ on line 8 is of the form \eqref{eq:special},
and
\[
\bar{\beta}_{i+1}=\mathrm{fl}(\|z\|)=\widetilde{\beta}_{i+1} = \beta_{i+1}
\]
resulting on line 10 to $\bar{v}_{i+1}=\mathrm{fl}(z/\beta_{i+1})=Pe_{i+1}=v_{i+1}$. 
\end{proof}\smallskip

Note that the same results can be shown also for the classical Gram-Schmidt
variant of Algorithm~\ref{alg:lanczos}, where we first compute $\alpha_{k}$
as $\alpha_{k}=v_{k}^{T}Av_{k}$, and then evaluate 
\[
z=Av_{k}-\alpha_{k}v_{k}-\beta_{k}v_{k-1}.
\]

The results of Theorem~\ref{thm:main} together with Lemma~\ref{lem:equiv}
indicate that the only nonzero structure of $A$ and $v$ that guarantees exact computations of the Lanczos algorithm
in floating point arithmetic is given by \eqref{eq:structure}. If
$A$ and $v$ do not have special structure \eqref{eq:structure},
the Lanczos algorithm can still compute exactly, but only in very
special cases where the particular input data are chosen such that
no rounding errors appear.

Theorem~\ref{thm:main} and Lemma~\ref{lem:equiv} can be analogously formulated for $A$ and $v$ with
$d < n$. In such case, instead of $P$ and $T$ we consider block diagonal matrices $\widetilde{P}$ and $\widetilde{T}$ of the form
\[
\widetilde{P}=\left[\begin{array}{cc}
P & 0\\
0 & R_1
\end{array}\right], \qquad
\widetilde{T}=\left[\begin{array}{cc}
T & 0\\
0 & R_2
\end{array}\right],
\]                 
where $P$ is a signed permutation matrix of size $d$, $T$ is a $d$ by $d$ tridiagonal matrix defined as in Theorem~\ref{thm:main} and Lemma~\ref{lem:equiv}, and $R_1,R_2$ are arbitrary square matrices of size $n-d$.

\section{The conjugate gradient method}
\label{sec:cg}

The results of the previous section motivate the question, whether
analogous results can be obtained also
for the conjugate gradient method that is closely related to the Lanczos
algorithm. 

Given a symmetric and positive definite (SPD) matrix $A\in\mathbb{R}^{n\times n}$
and a right-hand side vector $b\in\mathbb{R}^{n}$, we wish to solve
a system of linear algebraic equations 
\begin{align*}
Ax=b
\end{align*}
using the conjugate gradient method (CG). Consider first the classical
Hestenes and Stiefel variant of CG formulated in Algorithm~\ref{alg:cg}. 

\begin{algorithm}[h]
\caption{Conjugate gradients}
\label{alg:cg}

\begin{algorithmic}[0]

\STATE \textbf{input} $A$, $b$, $x_{0}$

\STATE $r_{0}=b-Ax_{0}$

\STATE $p_{0}=r_{0}$

\FOR{$k=1,\dots$ until convergence}

\STATE $\gamma_{k-1}=\frac{r_{k-1}^{T}r_{k-1}}{p_{k-1}^{T}Ap_{k-1}}$

\STATE $x_{k}=x_{k-1}+\gamma_{k-1}p_{k-1}$

\STATE $r_{k}=r_{k-1}-\gamma_{k-1}Ap_{k-1}$

\STATE $\delta_{k}=\frac{r_{k}^{T}r_{k}}{r_{k-1}^{T}r_{k-1}}$

\STATE $p_{k}=r_{k}+\delta_{k}p_{k-1}$

\ENDFOR

\end{algorithmic} 
\end{algorithm}

It is well-known that the vectors and the coefficients generated by CG and the Lanczos algorithm are closely related. In particular, if Algorithm~\ref{alg:lanczos}
is started with $A$ and $v=r_{0}$, then, in exact arithmetic,
%the
%Lanczos vectors determined by \Cref{alg:lanczos} and the residual
%vectors generated by \Cref{alg:cg} are related via 
\begin{equation}
v_{j+1}=(-1)^{j}\frac{r_{j}}{\|r_{j}\|}\,,\qquad j=0,\dots,k.\label{eq:CGLA}
\end{equation}

Let us recall that for $A$ and $v$ having the structure \eqref{eq:structure},
the Lanczos vectors $v_{j+1}$ are computed without any roundoff error, i.e,
they remain exactly orthogonal during finite precision
computations. Based on the relation \eqref{eq:CGLA} one could expect that the normalized CG residual vectors, computed by Algorithm~\ref{alg:cg} started with the same input data, will also be close to orthogonal. We now perform a numerical experiment showing that the orthogonality among CG residuals can be lost in general.

We consider the Strako\v{s} matrix \cite{St1991}, which is a diagonal matrix $\Lambda$ having the eigenvalues
\begin{equation}
\lambda_{i}=\lambda_{1}+\frac{i-1}{n-1}(\lambda_{n}-\lambda_{1})\rho^{n-i},\quad i=2,\ldots,n.\label{eq:strakos}
\end{equation}
In particular, we choose $n=24$, $\lambda_1= 10^{-3}$, $\lambda_n = 1$, $\rho = 0.7$ and define 
$v=[1,\ldots,1]^{T}$. To ensure that the results will closely approximate the results of
exact computations, we apply the Lanczos algorithm with double reorthogonalization to $\Lambda$ and $v$.
In the last iteration we obtain the
symmetric tridiagonal matrix $\bar{T}_{n}$ having (almost) the same spectrum as
$\Lambda$.

Define $x_{0}\equiv0$, $A\equiv \bar{T}_{n}$ and $b\equiv e_{1}$ so that 
the input data $A$ and $b$ for the CG algorithm have the desired structure \eqref{eq:structure}.
Theorem~\ref{thm:main} ensures
that Algorithm~\ref{alg:lanczos} applied to $A$ and $b$ computes
exactly. 
However, Figure~\ref{fig:2} demonstrates that this is no more true for Algorithm~\ref{alg:cg}.

In Figure~\ref{fig:2} we plot the loss of orthogonality among the
normalized residual vectors
\[
\widetilde{v}_{j+1}\equiv(-1)^{j}\frac{\bar{r}_{j}}{\|\bar{r}_{j}\|}\,,\qquad j=0,\dots,k,
\]
computed by Algorithm~\ref{alg:cg}. The loss of orthogonality
is measured using the quantity
\[
\left\Vert \widetilde{V}_{k}^{T}\widetilde{V}_{k}-I\right\Vert _{F},
\]
where $\widetilde{V}_{k} = [\widetilde{v}_1, \ldots, \widetilde{v}_k]$.
We observe that the orthogonality is lost quickly. As a consequence, the Hestenes and Stiefel version of CG (Algorithm~\ref{alg:cg}) does not compute exactly and rounding errors influence significantly the performance of the algorithm. 
\begin{figure}
\begin{centering}
\includegraphics[width=0.6\textwidth]{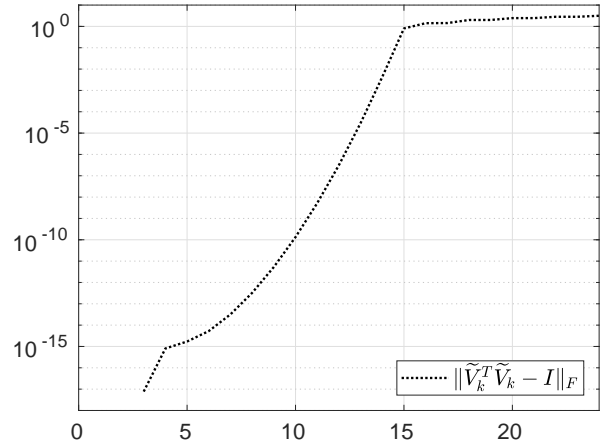} 
\par\end{centering}
\caption{Loss of orthogonality measured by $\|\widetilde{V}_{k}^{T}\widetilde{V}_{k}-I\|_{F}$
in Algorithm~\ref{alg:cg}. }
\label{fig:2} 
\end{figure}

We did not find a nonzero structure of the input data $A$ and $b$
such that Algorithm~\ref{alg:cg} computes (almost) exactly. Since the
coefficients $\gamma_{k-1}$ and $\delta_{k}$ are ratios of two floating
point numbers, it is very unlikely that such a structure exists. Nevertheless,
we can use the knowledge about the exact computations of the Lanczos
algorithm and the close relationship between both algorithms to develop
an algorithmic version of CG that computes ``almost exactly'' for
the input data having the structure \eqref{eq:structure}. The idea
is simply to compute the exact Lanczos vectors and reconstruct the
CG quantities from the Lanczos vectors. Sometimes,
this variant of the CG method is denoted as the cgLanczos algorithm; see \cite{PaSa1975}.

By comparing the corresponding recurrences for computing the Lanczos
vectors $v_{j+1}$ (Algorithm~\ref{alg:lanczos}) and the CG residual vectors
$r_{j}$ (Algorithm~\ref{alg:cg}), and using \eqref{eq:CGLA} one can find
the relationship among the Lanczos and CG coefficients 
\begin{equation}
\beta_{k+1}=\frac{\sqrt{\delta_{k}}}{\gamma_{k-1}},\quad\alpha_{k}=\frac{1}{\gamma_{k-1}}+\frac{\delta_{k-1}}{\gamma_{k-2}},\quad\delta_{0}=0,\quad\gamma_{-1}=1.\label{eq:CGLanczos}
\end{equation}
Writing \eqref{eq:CGLanczos} in the matrix form we find out that
CG computes implicitly the $LDL^{T}$ factorization of $T_{k}$ 
\begin{equation}
T_{k}  =  \left[\begin{array}{cccc}
1\\
\ell_{1} & \ddots\\
 & \ddots & \ddots\\
 &  & \ell_{k-1} & 1
\end{array}\right]\left[\begin{array}{cccc}
d_{1}\\
 & \ddots\\
 &  & \ddots\\
 &  &  & d_{k}
\end{array}\right]\left[\begin{array}{cccc}
1 & \ell_{1}\\
 & \ddots & \ddots\\
 &  & \ddots & \ell_{k-1}\\
 &  &  & 1
\end{array}\right],\label{eq:matrixLDL}
\end{equation}
where 
\[
\ell_{j}=\sqrt{\delta_{j}},\quad j=1,\dots,k-1,\quad\mbox{and}\quad  d_{j}=\gamma_{j-1}^{-1},\quad j=1,\dots,k,
\]
are easily expressible from the CG coefficients. Therefore, knowing
$T_{k}$, we can compute its $LDL^{T}$ factorization to reconstruct
the CG coefficients. The factorization can be computed using 
\begin{equation}
d_{1}=\alpha_{1},\quad\ell_{j}=\frac{\beta_{j+1}}{d_{j}},\quad d_{j+1}=\alpha_{j+1}-\beta_{j+1}\ell_{j},\qquad j=1,\dots,k-1;\label{eq:LDL}
\end{equation}
see, e.g., \cite[p.25]{B:GoMe2010}.

Suppose now that the Lanczos vectors and coefficients are known. Assuming for simplicity $x_0=0$, we
would like to reconstruct the CG approximate solutions $x_{k}$ from
the Lanczos process. It is well-known that 
\begin{equation}
x_{k}= V_{k}y_{k},\qquad T_{k}y_{k}=\|b\|e_{1}.\label{eq:CGcomputed}
\end{equation}
In the special case of the input data having the structure \eqref{eq:structure}
one can assume that $T_{k}\in\mathbb{F}^{k\times k}$ and $\|b\|\in\mathbb{F}$
are computed exactly using Algorithm~\ref{alg:lanczos}. If we are able to compute the solution of the
system $T_{k}y_{k}=\|b\|e_{1}$ exactly, then $x_{k}= V_{k}y_{k}$
would be the exact CG approximation since columns of $V_{k}$ are just plus or
minus columns of the identity matrix. However, in general, 
the system $T_{k}y_{k}=\|b\|e_{1}$ has to be solved numerically and only the computed
solution $\bar{y}_{k}$ is available.

Using \cite[Theorem 9.14, p.~176]{B:Hi2002}, the numerical solution $\bar{y}_{k}$
of the system with tridiagonal symmetric and positive definite $T_{k}$
computed using the $LDL^{T}$ factorization of $T_{k}$ is the exact
solution of the perturbed problem
\[
\left(T_{k}+\Delta\right)\bar{y}_{k}=\|b\|e_{1},\qquad|\Delta|\leq 5\ru|T_{k}|= 5\ru T_{k}.
\]
Therefore, 
%having the input
%data with the structure \eqref{eq:structure}, we obtain
\[
y_{k} = (I + T_{k}^{-1}\Delta)\bar{y}_{k}.
\]
so that
\[
x_{k}-\bar{x}_{k}=V_{k}(y_{k}-\bar{y}_{k})
%=V_{k}T_{k}^{-1}\left(T_{k}y_{k}-T_{k}\bar{y}_{k}\right)
=V_{k}T_{k}^{-1}\Delta\bar{y}_{k}.
\]
Assuming that $5\ru\kappa(A) < 1$, we get 
$$\norm{T_{k}^{-1}\Delta} \leq 5\ru\kappa(T_{k})\leq 5\ru\kappa(A)< 1\,.
$$
Hence, 
$I + T_{k}^{-1}\Delta$ is nonsingular and
%\[
%x_{k}-\bar{x}_{k}=V_{k}T_{k}^{-1}\Delta(I + T_{k}^{-1}\Delta)^{-1}y_{k},
%\]
%and
\[
\norm{(I + T_{k}^{-1}\Delta)^{-1}} \leq \frac{1}{1 - \norm{T_{k}^{-1}\Delta}}.
\]
Finally, using $\left\Vert x_{k}\right\Vert =\left\Vert V_{k}y_{k}\right\Vert =\left\Vert y_{k}\right\Vert$ we obtain 
\begin{eqnarray*}
\frac{\left\Vert x_{k}-\bar{x}_{k}\right\Vert }{\left\Vert x_{k}\right\Vert }&=& 
\frac{\left\Vert T_{k}^{-1}\Delta \bar{y}_k\right\Vert}{\left\Vert y_{k}\right\Vert }=
\frac{\left\Vert T_{k}^{-1}\Delta(I + T_{k}^{-1}\Delta)^{-1}y_{k}\right\Vert}{\left\Vert y_{k}\right\Vert }\\
&\leq&\frac{5\ru\kappa(T_{k})}{1-5\ru\kappa(T_{k})}\leq\frac{5\ru\kappa(A)}{1-5\ru\kappa(A)}.
\end{eqnarray*}
The results are summarized in the following theorem.\smallskip

\begin{theorem}\label{thm:main-1}Let a symmetric and positive definite
matrix $A$ and a vector $b$ have the structure \eqref{eq:structure}.
Suppose that $V_{k}$ and $T_{k}$ are computed using the Lanczos
algorithm (Algorithm~\ref{alg:lanczos}) applied to $A$ and $b$, and
that the system $T_{k}y_{k}=\|b\|e_{1}$ is solved numerically
using $LDL^{T}$ factorization giving the computed solution $\bar{y}_{k}$.
Let $x_0=0$.
Then, under the assumption $5\ru\kappa(A) < 1$, the computed CG approximate solution $\bar{x}_{k}=V_{k}\bar{y}_{k}$, $k>0$, 
satisfies 
\begin{equation}
\frac{\left\Vert x_{k}-\bar{x}_{k}\right\Vert }{\left\Vert x_{k}\right\Vert }\leq\frac{5\ru\kappa(A)}{1-5\ru\kappa(A)},
\label{eq:CGexact}
\end{equation}
where $x_{k}$ is the exact CG approximation.
\end{theorem} \smallskip

The above results demonstrate that almost exact CG approximate solutions
can be computed without reorthogonalization. Naturally, the above
mentioned version of CG is not too efficient since it requires storing
the Lanczos vectors $V_{k}$ and the matrix $T_{k}$. Below we derive
a more efficient version of CG that preserves the above idea: first
compute the Lanczos vectors and coefficients and then reconstruct
the CG related quantities. Using

\[
\|r_{k}\|=\sqrt{\delta_{k}\delta_{k-1}\dots\delta_{1}}\|r_{0}\|=\ell_{1}\dots\ell_{k}\|r_{0}\|
\]
we obtain 
\begin{eqnarray}
r_{k} & = & (-1)^{k}\|r_{k}\|\,v_{k+1}=(-1)^{k}\|r_{0}\|\ell_{1}\dots\ell_{k}\,v_{k+1},\label{eq:CGr}\\
p_{k} & = & r_{k}+\delta_{k}p_{k-1}=r_{k}+\ell_{k}^{2}p_{k-1},\\
x_{k} & = & x_{k-1}+\gamma_{k-1}p_{k-1}=x_{k-1}+\frac{p_{k-1}}{d_{k}}.\label{eq:CGx}
\end{eqnarray}
The final cgLanczos algorithm is given by Algorithm~\ref{alg:lanczos-cg}.
For simplicity we choose $x_{0}=0$ so that $r_{0}=b$. 
Note that the cgLanczos algorithm follows 
 in a straightforward way from the results of
\cite[Section~4]{PaSa1975}.
\begin{algorithm}[ht]
\caption{cgLanczos algorithm}

\label{alg:lanczos-cg}

\begin{algorithmic}[0]

\STATE \textbf{input} $A$, $b$

\STATE $\beta_{1}=0$, $v_{0}=0$, $\ell_{0}=0$, $x_{0}=0$

\STATE $r_{0}=b$, $p_{0}=r_{0}$,

\STATE $\rho_{0}=\|b\|$

\STATE $v_{1}=b/\rho_{0}$

\FOR{$k=1,\dots$}

\STATE $w=Av_{k}-\beta_{k}v_{k-1}$

\STATE $\alpha_{k}=w^{T}v_{k}$

\STATE $w=w-\alpha_{k}v_{k}$\hspace*{4em}\rlap{\smash{$\left.\begin{array}{@{}c@{}}
\\
{}\\
{}\\
{}
\end{array}\right\} \begin{tabular}{l}
 \ensuremath{T_{k}} and \ensuremath{V_{k}}\end{tabular}$}}

\STATE $\beta_{k+1}=\|w\|$

\STATE $v_{k+1}=w/\beta_{k+1}$

\STATE $d_{k}=\alpha_{k}-\beta_{k}\ell_{k-1}$

\STATE $\ell_{k}=\frac{\beta_{k+1}}{d_{k}}$\hspace*{6em}\rlap{\smash{$\left.\begin{array}{@{}c@{}}
\\
{}
\end{array}\right\} \begin{tabular}{l}
 \ensuremath{T_{k}=L_{k}D_{k}L_{k}^{T}}\end{tabular}$}}

\STATE $\rho_{k}=\ell_{k}\rho_{k-1}$

\STATE $x_{k}=x_{k-1}+\frac{p_{k-1}}{d_{k}}$

\STATE $r_{k}=(-1)^{k}\rho_{k}v_{k+1}$\hspace*{2.5em}\rlap{\smash{$\left.\begin{array}{@{}c@{}}
\\
{}
\end{array}\right\} \begin{tabular}{l}
 vectors \ensuremath{x_{k}}, \ensuremath{r_{k}} and \ensuremath{p_{k}}\end{tabular}$}}

\STATE $p_{k}=r_{k}+\ell_{k}^{2}p_{k-1}$

\ENDFOR

\end{algorithmic} 
\end{algorithm}

The Algorithm~\ref{alg:lanczos-cg} has three parts marked out by brackets.
First, the Lanczos vectors and coefficients are computed as in Algorithm~\ref{alg:lanczos}.
In the second part the algorithm computes the $LDL^{T}$ factorization
via \eqref{eq:LDL} and the last part computes the CG vectors $p_{j}$,
$r_{j}$ and $x_{j}$ using \eqref{eq:CGr}-\eqref{eq:CGx}. We can
see immediately that if we apply Algorithm~\ref{alg:lanczos-cg} to $A$
and $b$ having the structure \eqref{eq:structure}, the residual
vectors are exactly orthogonal during finite precision computations
as in the case of Algorithm~\ref{alg:lanczos}. The computed coefficients $\bar{\ell}_{j}$
and $\bar{d}_{j}$ are almost exact in the sense
\[
T_{k}+\Delta=\bar{L}_{k}\bar{D}_{k}\bar{L}_{k}^{T},\qquad|\Delta|\leq5\ru T_{k},
\]
see \cite[p.~174]{B:Hi2002}, where $\bar{L}_{k}$ and $\bar{D}_{k}$
are the computed factors of the $LDL^{T}$ factorization of $T_{k}$.
Therefore, one can expect that the CG approximate solution $\bar{x}_{k}$ computed
using Algorithm~\ref{alg:lanczos-cg} will satisfy the relation \eqref{eq:CGexact}.

For numerical demonstration we consider the same problem as at the beginning
of this section, i.e., we consider $A$ and $b$ having the structure \eqref{eq:structure},
that have been obtained from the Lanczos algorithm with double reorthogonalization applied to $\Lambda$ and $v$. However, instead 
of Hestenes and Stiefel version of CG (Algorithm~\ref{alg:cg})
we apply the cgLanczos algorithm (Algorithm~\ref{alg:lanczos-cg})
to solve the system $Ax=b$ with $x_0=0$. It is clear that residuals must be exactly orthogonal. 
Hence, we measure the quality of results computed 
by Algorithm~\ref{alg:lanczos-cg} using the $A$-orthogonality
of the reconstructed direction vectors, and using the relative distance between the exact and the computed CG approximations. 
\begin{figure}
\begin{centering}
\includegraphics[width=0.6\textwidth]{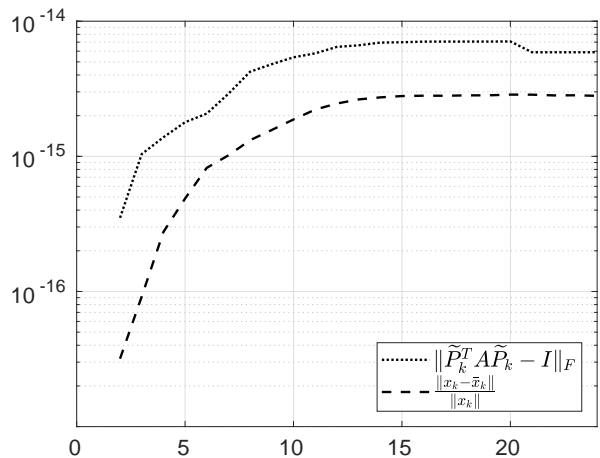} 
\par\end{centering}
\caption{The loss of $A$-orthogonality among direction vectors computed by Algorithm~\ref{alg:lanczos-cg}. }
\label{fig:3} 
\end{figure}

In Figure~\ref{fig:3} we plot the loss of $A$-orthogonality (dotted curve) among the normalized direction
vectors 
$$
\widetilde{p}_{k} = \frac{\bar{p}_{k}}{\Vert \bar{p}_{k} \Vert_A}
$$
computed by Algorithm~\ref{alg:lanczos-cg}. The loss of
$A$-orthogonality is measured by the Frobenius norm of the matrix
$\widetilde{P}_{k}^{T}A\widetilde{P}_{k}-I$,
%\[
%\norm{\hat{P}_{k}^{T}A\hat{P}_{k}-I}_{F}
%\]
where $\widetilde{P}_{k}=[\widetilde{p}_{0},\dots,\widetilde{p}_{k-1}]$. As expected,
the loss of $A$-orthogonality is close to the machine
precision level. Moreover, we also plot the quantity
$$
\frac{\left\Vert x_{k}-\bar{x}_{k}\right\Vert }{\left\Vert {x}_{k}\right\Vert }
$$
(dashed curve), where $\bar{x}_{k}$ were computed in double precision using Algorithm~\ref{alg:lanczos-cg} and the exact approximations $x_{k}$ were computed using Algorithm~\ref{alg:lanczos-cg} in extended precision arithmetic with 128 valid digits (Matlab's vpa arithmetic). As expected and predicted by 
Theorem~\ref{thm:main-1}, the relative error is close to the machine precision level. Note that $\kappa(A)=10^3$.

\section{Other algorithms}
\label{sec:other}
In Section~\ref{sec:flan} we parametrized matrices $A$ and starting vectors $v$ that guarantee exact computations of the Lanczos algorithm.
In this section we demonstrate that the ideas of 
Section~\ref{sec:flan} can be generalized to other algorithms
for computing bases of Krylov subspaces. In particular, if $A$ is not symmetric, we can use the Arnoldi algorithm \cite{Ar1952} for computing the orthonormal basis, 
or the nonsymmetric Lanczos algorithm \cite{La1950}
for computing the bi-orthogonal basis. When working with Krylov subspaces generated by symmetric matrices $A^TA$ or $AA^T$, one can consider the Golub-Kahan bidiagonalization \cite{GoKa1965}. The ideas can be further generalized to block Krylov subspaces method like the block-Lanczos \cite{GoUn1977} or block-Arnoldi algorithms. 
We will show that there exists a nonzero structure of the input data that guarantees exact computations of the above mentioned algorithms.
For each algorithm we define the index $d$ that corresponds to the maximal dimension of the corresponding subspaces, and formulate the
final results for $d=n$. Nevertheless, all results can be generalized  to the case $d < n$ similarly as for the Lanczos algorithm; see Section~\ref{sec:flan}. 

\subsection{Arnoldi algorithm}
\label{subsec:Arnoldi}
A natural generalization of the Lanczos algorithm for nonsymmetric matrices is the Arnoldi algorithm; see \cite{Ar1952}. Given a square matrix $A\in\mathbb{R}^{n \times n}$ and assuming $k<d=d(A,v)$, the Arnoldi algorithm (Algorithm~\ref{alg:Arnoldi}) computes an orthonormal basis $v_1,\ldots,v_{k+1}$ of the Krylov subspace $\mathcal{K}_{k+1}(A,v)$.
\begin{algorithm}[!ht]
\caption{Arnoldi algorithm} 
\label{alg:Arnoldi}
\begin{algorithmic}[0]

\STATE \textbf{input} $A$, $v$

\STATE $ v_1 = v/\norm{v} $

\FOR{$j=1,\dots, k$}

\STATE $w = A v_j$

\FOR{$i = 1 : j$}

\STATE $h_{i,j} = v_i^T w$

\STATE $w = w - h_{i,j} v_i$

\ENDFOR

\STATE $h_{j+1,j} = \norm{w}$

\STATE $v_{j+1} = \frac{w}{h_{j+1,j}}$

\ENDFOR
\end{algorithmic}
\end{algorithm} 
The computed vectors and coefficients satisfy
$$
AV_k = V_kH_k + h_{k+1,k}v_{k+1}e_k^T,
$$
where $V_k = [v_1, \dots, v_k]$ and
$$
H_k=\left[\begin{array}{cccc}
h_{1,1} & \ldots & \ldots & h_{1,k} \\
h_{2,1} & \ddots & & \vdots \\
 & \ddots & \ddots & \vdots  \\
 &  & h_{k,k-1} & h_{k,k}
\end{array}\right]
$$
is upper Hessenberg 
with $h_{j+1,j} > 0, j = 1,\hdots,k-1$. Note that if $A$ symmetric, then $H_k$ is symmetric and tridiagonal, and 
Algorithm~\ref{alg:Arnoldi} is equivalent to Algorithm~\ref{alg:lanczos}.

% The following lemma is generalization of Lemma~\ref{lem:equiv} for the Arnoldi algorithm.

% \begin{lemma}
% Assuming exact arithmetic, Algorithm~\ref{alg:Arnoldi} applied to symmetric $A\in\mathbb{R}^{n \times n}$ and $v\in\mathbb{R}^n$ such that $d=n$ produces sequence of vectors $v_1,v_2,\ldots$ equal to plus or minus columns of identity matrix if and only if
% $$
% A = PHP^T,\qquad v = hPe_1,
% $$
% with $P\in\mathbb{R}^{n \times n}$ being a signed permutation matrix and $H\in\mathbb{R}^{n \times n}$ being an upper Hessenberg matrix of the form
% $$
% H=\left[\begin{array}{cccc}
% \widetilde{h}_{1,1} & \ldots & \ldots & \widetilde{h}_{1,n} \\
% \widetilde{h}_{2,1} & \ddots & & \vdots \\
%  & \ddots & \ddots & \vdots  \\
%  &  & \widetilde{h}_{n,n-1} & \widetilde{h}_{n,n}
% \end{array}\right], 
% $$ 
% where $\widetilde{h}_{j+1,j},h > 0$, $j = 1,\ldots,n-1$. Moreover, the upper Hessenberg matrix $H_{n}$ \cred{resulting from Algorithm~\ref{alg:Arnoldi}} is equal to $H$.
% \label{lemma-arnoldi}
% \end{lemma}
Theorem~\ref{thm:main} for the Lanczos algorithm can now be generalized in a straightforward way for the Arnoldi algorithm. We state the corresponding theorem without a proof.
\smallskip

\begin{theorem}
Consider the standard model of floating point arithmetic. Let
$$
A = PHP^T, \qquad v=\widetilde{h}_{1,0}Pe_1,
$$
where $P\in\mathbb{F}^{n \times n}$ is a signed permutation matrix and
$$
H=\left[\begin{array}{cccc}
\widetilde{h}_{1,1} & \ldots & \ldots & \widetilde{h}_{1,n} \\
\widetilde{h}_{2,1} & \ddots & & \vdots \\
 & \ddots & \ddots & \vdots  \\
 &  & \widetilde{h}_{n,n-1} & \widetilde{h}_{n,n}
\end{array}\right] 
$$ 
with $\widetilde{h}_{j+1,j}>0$ and $\widetilde{h}^2_{j+1,j}$ within the exponent range,
%$\mathtt{realmin} \leq \widetilde{h}^2_{j+1,j}\leq \mathtt{realmax}$, 
$j = 0,\ldots,n-1$. 
Then Algorithm~\ref{alg:Arnoldi} applied to $A$ and $v$ computes exactly. As a consequence, it holds that $H_{n} = H$.
\label{thm-arnoldi}
\end{theorem}

%Lemma~\ref{lemma-arnoldi} and 
%Theorem~\ref{thm-arnoldi} can also be generalized for $d < n$ similarly as in the case of the Lanczos algorithm; see Section~\ref{sec:flan}.

\subsection{Nonsymmetric Lanczos algorithm}
\label{subsec:nlanczos}

Given $A\in\mathbb{R}^{n \times n}$ and $v,w\in\mathbb{R}^{n}$, such that $w^Tv \neq 0$, we denote 
\begin{equation*}
    d = \min(d(A,v),d(A^T,w)).    
\end{equation*}
Assuming $k<d$ and $\beta_{i+1}\neq 0$, $i = 1,\ldots,k$, the nonsymmetric Lanczos algorithm \cite{La1950} (Algorithm~\ref{alg:nlanczos}) computes 
\begin{algorithm}[th]
\caption{nonsymmetric Lanczos algorithm}
\label{alg:nlanczos}

\begin{algorithmic}[0]

\STATE \textbf{input} $A$, $v$,  $w$

\STATE $v_{0}=w_0 = 0$, $\gamma_1=\|v \|$, $v_{1}=v/\gamma_1$

\STATE $\beta_{1} = w^Tv_1$, $w_{1}=w/\beta_{1}$  

\FOR{$i=1,\dots,k$} 

\STATE $\alpha_{i}=w_i^{T}A v_{i}$ 

\STATE $v =  Av_{i}-\alpha_i v_i - \beta_{i}v_{i-1}$ 

\STATE $\gamma_{i+1} = \| v\|$

\STATE $v_{i+1}=v/\gamma_{i+1}$

\STATE $w = A^Tw_{i}-\alpha_i w_i - \gamma_{i}w_{i-1}$

\STATE $\beta_{i+1} = v_{i+1}^T w$ 
\STATE $w_{i+1} = w/\beta_{i+1}$ 

\ENDFOR 

\end{algorithmic} 
\end{algorithm}
two sets $v_1,\dots,v_{k+1}$ and $w_1,\dots,w_{k+1}$ of bi-orthogonal vectors. The vectors and coefficients generated by Algorithm~\ref{alg:nlanczos} satisfy
\begin{eqnarray*}
AV_k &=& V_k T_k + \gamma_{k+1} v_{k+1} e_k^T, \\
A^T W_k &=& W_k T_k^T + \beta_{k+1} w_{k+1} e_k^T, \\
W_k^T V_k &=& I,
\end{eqnarray*}
where $V_k = [v_1,\dots,v_k]\in \mathbb{R}^{n \times k}$, $W_k = [w_1,\dots,w_k]\in \mathbb{R}^{n \times k}$,
and
\begin{equation*}
T_{k}=\left[\begin{array}{cccc}
\alpha_{1} & \beta_{2}\\
\gamma_{2} & \ddots & \ddots\\
 & \ddots & \ddots & \beta_{k}\\
 &  & \gamma_{k} & \alpha_{k}
\end{array}\right].
\end{equation*}

The nonsymmetric Lanczos algorithm is based on two three-term reccurences similar to the reccurence from the Lanczos algorithm. Using the same 
technique as for the Lanczos algorithm,
we obtain 
an analogy of Theorem~\ref{thm:main}
that we present
without a proof.
%It is therefore quite natural that we can formulate 
%an analogy of the Theorem~\ref{thm:main} \cblue{also} for the nonsymmetric Lanczos algorithm. We present the following theorem without proof.
\smallskip

\begin{theorem}
Consider the standard model of floating point arithmetic. Let
$$
A = PTP^T, \quad v=\widetilde{\gamma}_1Pe_1, \quad w=\widetilde{\beta}_1Pe_1,
$$
where $P\in\mathbb{F}^{n \times n}$ is a sign permutation matrix and $T\in\mathbb{F}^{n \times n}$ is tridiagonal of the form
$$
T=\left[\begin{array}{cccc}
\widetilde{\alpha}_{1} & \widetilde{\beta}_{2}\\
\widetilde{\gamma}_{2} & \ddots & \ddots\\
 & \ddots & \ddots & \widetilde{\beta}_{n}\\
 &  & \widetilde{\gamma}_{n} & \widetilde{\alpha}_{n}
\end{array}\right], 
$$ 
with $\widetilde{\beta}_j \neq 0$, $\widetilde{\gamma}_j> 0$ and 
$\widetilde{\gamma}^2_j$ within the exponent range,
%$ \mathtt{realmin} \leq \widetilde{\gamma}^2_j\leq \mathtt{realmax}$, 
$j = 1,\ldots,n$. Then Algorithm~\ref{alg:nlanczos} applied to $A$, $v$ and $w$ computes exactly. As a consequence, it holds that $T_{n} = T$.
\label{thm-nlanczos}
\end{theorem}

%Note that Theorem~\ref{thm-nlanczos} can also be generalized 
%for $d < n$, similarly as for the Lanczos algorithm; see Section~\ref{sec:flan}. 

\subsection{Golub-Kahan bidiagonalization}
\label{subsec:GKbidiag}
Let $A \in \mathbb{R}^{n \times m}$, $v \in \mathbb{R}^n$, and denote
\begin{equation*}
    d = \min(d(AA^T,v),d(A^TA,A^Tv)).    
\end{equation*}
Assuming $k<d$, the Golub-Kahan bidiagonalization \cite{GoKa1965}
(Algorithm~\ref{alg:bidiagonal})
generates two sets of orthonormal vectors $s_1, \ldots, s_{k+1}$ and $w_1, \ldots, w_k$. 
The coefficients 
$\gamma_i$ and $\delta_{i+1}$ 
that appear in Algorithm~\ref{alg:bidiagonal}
are normalization coefficients.
\begin{algorithm}[ht]
\caption{Golub-Kahan bidiagonalization}

\label{alg:bidiagonal}

\begin{algorithmic}[0]

\STATE \textbf{input} $A$, $v$

\STATE $w_0 = 0$
\STATE $\delta_{1}s_1 =v$

\FOR{$i=1,\dots,k$}

\STATE $\gamma_i w_i = A^Ts_i - \delta_i w_{i-1}$

\STATE $\delta_{i+1}s_{i+1} = Aw_i - \gamma_i s_i$

\ENDFOR

\end{algorithmic}
\end{algorithm}

Denoting 
$S_k = [s_1,\dots,s_k]\in \mathbb{R}^{n \times k}$ 
and $W_k=[w_1,\dots,w_k] \in\mathbb{R}^{m \times k}$,
the vectors and coefficients generated by 
Algorithm~\ref{alg:bidiagonal}
satisfy
\begin{align*}
A^T S_k &= W_k L_k^T, \\
AW_k &= S_kL_k + s_{k+1}\delta_{k+1}e_k^T,
\end{align*}
where
\begin{equation*}
L_{k}=\left[\begin{array}{cccc}
\gamma_{1}\\
\delta_{2} & \gamma_{2}\\
 & \ddots & \ddots & \\
 &  & \delta_{k} & \gamma_{k}
\end{array}\right].
\end{equation*}
Under the assumption $k<d$, the coefficients $\gamma_i$'s as well as $\delta_i$'s are positive, $i=1,\dots,k$.

It is well known that the Golub-Kahan bidiagonalization is closely related to the Lanczos algorithm. In more detail, the orthonormal columns of $S_k$
can be seen as the Lanczos vectors generated by 
$AA^T$ with the starting vector $v$. Similarly, 
$W_k$ contains the Lanczos vectors generated by 
$A^TA$ and $A^Tv$. Therefore, it is not surprising that the results of 
Section~\ref{sec:flan}
for the Lanczos algorithm can be analogously  formulated also for the Golub-Kahan bidiagonalization.
We present here (without a proof) an analogy 
%of Lemma~\ref{lem:equiv} and 
of Theorem~\ref{thm:main} formulated for $A\in{\mathbb F}^{n\times n}$.\smallskip

% \begin{lemma}
% Assuming exact arithmetic, Algorithm~\ref{alg:bidiagonal} applied to  $A\in\mathbb{R}^{n \times n}$ and $v\in\mathbb{R}^n$ with $d=n$, where $d$ is defined by \eqref{eq:degbi},  produces sequences of vectors $s_i$ and $w_i$ equal to plus or minus columns of identity matrix if and only if
% $$
% A = PLP^T,\qquad v = \widetilde{\delta}_1Pe_1,
% $$
% with $P\in\mathbb{R}^{n \times n}$ being a signed permutation matrix and $L\in\mathbb{R}^{n \times n}$ being bidiagonal,
% $$
% L=\left[\begin{array}{cccc}
% \widetilde{\gamma}_{1}\\
% \widetilde{\delta}_{2} & \widetilde{\gamma}_{2}\\
%  & \ddots & \ddots & \\
%  &  & \widetilde{\delta}_{n} & \widetilde{\gamma}_{n}
% \end{array}\right], 
% $$ 
% where $\widetilde{\gamma}_j,\widetilde{\delta}_j> 0$, $j = 1,\ldots,n$. Moreover, $L_{n}$ resulting from Algorithm~\ref{alg:bidiagonal} satisfies $L_n=L$.
% \label{lemma-bidiag}
% \end{lemma}\smallskip

% Now we can use Lemma~\ref{lemma-bidiag} to parametrize inputs, for which the Golub-Kahan bidiagonalization computes exactly in finite precision arithmetic. \smallskip

\begin{theorem}
Consider the standard model of floating point arithmetic. Let
$$
A = PLP^T, \qquad v=\widetilde{\delta}_1Pe_1,
$$
where $P\in\mathbb{F}^{n \times n}$ is a sign permutation matrix and $L\in\mathbb{F}^{n \times n}$ is bidiagonal of the form
$$
L=\left[\begin{array}{cccc}
\widetilde{\gamma}_{1}\\
\widetilde{\delta}_{2} & \widetilde{\gamma}_{2}\\
 & \ddots & \ddots & \\
 &  & \widetilde{\delta}_{n} & \widetilde{\gamma}_{n}
\end{array}\right], 
$$ 
with $\widetilde{\gamma}_j,\widetilde{\delta}_j> 0$ and 
$\widetilde{\gamma}^2_j,\widetilde{\delta}^2_j$ within the exponent range.
%$ \mathtt{realmin} \leq \widetilde{\gamma}^2_j,\widetilde{\delta}^2_j\leq \mathtt{realmax}$, 
$j = 1,\ldots,n$. Then Algorithm~\ref{alg:bidiagonal} applied to $A$ and $v$ computes exactly. As a consequence, it holds that $L_{n} = L$.
\label{thm-bidiag}
\end{theorem}

%Note that Theorem~\ref{thm-bidiag} can also be generalized 
%for $d < n$, similarly as for the Lanczos algorithm; see Section~\ref{sec:flan}. 

\subsection{Block Lanczos algorithm}
\label{subsec:blockLanc}

The Lanczos algorithm has also an analogy for block matrices known as the block-Lanczos algorithm; see \cite{GoUn1977}. Given a block symmetric matrix $A\in\mathbb{R}^{n \times n}$ with $p$ by $p$ blocks, i.e., $n = mp$ for some $m\in\mathbb{N}$, and a block vector $U_1\in\mathbb{R}^{n \times p}$, we can define a sequence of block Krylov subspaces 
\[
\mathcal{K}_{k}(A,U_1)=\mathrm{colspan}\{U_1,AU_1,\dots,A^{k-1}U_1\}
\]
and denote the maximal achievable dimension of these nested subspaces as $d = d(A,U_1)$. 

Let $I$ denote the $p$ by $p$ identity matrix and let $0$ denote the $p$ by $p$ zero matrix. Let $U_1$ has orthonormal columns and $d=\mathcal{K}_{m}(A,U_1)=n$. Assuming $k < m$, the block Lanczos algorithm (Algorithm~\ref{alg:blockLanc})
\begin{algorithm}[ht]
\caption{block Lanczos algorithm}

\label{alg:blockLanc}

\begin{algorithmic}[0]

\STATE \textbf{input} $A\in\mathbb{R}^{n \times n}$, $U_1\in\mathbb{R}^{n \times p}$ such that $U_1^TU_1 = I$

\STATE $U_0 = U_1$, $B_1 = 0$

\STATE $M_1 = U_1^TAU_1$

\FOR{$i=1,\dots,k$}

\STATE $R_{i+1} = AU_i - U_iM_i -  U_{i-1}B^T_i$

\STATE $R_{i+1} = U_{i+1}B_{i+1}$ (QR factorization of $R_{i+1}$)

\STATE $M_{i+1} = U_{i+1}^TAU_{i+1}$

\ENDFOR

\end{algorithmic}
\end{algorithm}
generates an orthonormal sequence of block vectors $U_{i}\in\mathbb{R}^{n \times p}$, i.e., $U_i^TU_j = \delta_{i,j}I$ ($\delta_{i,j}$ denotes Kronecker delta), satisfying the relation 
$$ 
A\left[\:U_1,\ldots, U_k\:\right] = \left[\:U_1,\ldots, U_k\:\right]T_k + R_{k+1}\left[\:0,\ldots,0, I\:\right],
$$
where
$$
T_k = \left[\begin{array}{cccc}
M_1 & B_2^T\\
B_{2} & \ddots & \ddots \\
 & \ddots & \ddots & B^T_k  \\
 &  & B_k & M_k
\end{array}\right]
$$
is a block tridiagonal matrix. The blocks $M_j \in\mathbb{R}^{p \times p}$, $j = 1,\ldots,k$, are symmetric matrices and $B_{j+1}\in\mathbb{R}^{p \times p}$, $j = 1,\ldots,k-1$, are upper triangular matrices.

Further, we define the \textit{signed block permutation} matrix as a square block matrix with only one nonzero block in each block row and block column, where the nonzero blocks are sign permutation matrices. 
We now present an analogy of Theorem~\ref{thm:main}.
\smallskip

\begin{theorem}
Consider the standard model of floating point arithmetic.
Let $n = mp$ for $n,m,p\in\mathbb{N}$, and let
$$
A = PTP^T, \quad U_1=P\left[\:I, 0,\ldots, 0\:\right]^T,
$$
where $P\in\mathbb{F}^{n \times n}$ is a signed block permutation matrix with blocks of size $p$, $I,0\in\mathbb{F}^{p \times p}$, and $T\in\mathbb{F}^{n \times n}$ is a block tridiagonal matrix of the form
$$
T=\left[\begin{array}{cccc}
\widetilde{M}_1 & \widetilde{B}_2^T\\
\widetilde{B}_{2} & \ddots & \ddots\\
 & \ddots & \ddots & \widetilde{B}^T_m  \\
 &  & \widetilde{B}_m & \widetilde{M}_m
\end{array}\right], 
$$ 
where $\widetilde{M}_i\in\mathbb{F}^{p \times p}$ are symmetric and $\widetilde{B}_{i+1}\in\mathbb{F}^{p \times p}$ are upper triangular with positive entries on the diagonal. Assume that the QR factorization 
in Algorithm~\ref{alg:blockLanc} is computed 
using the classical or modified Gram-Schmidt algorithm without any underflow or overflow.
Then Algorithm~\ref{alg:blockLanc} applied to $A$ and $U_1$ computes exactly. As a consequence, it holds that $T_{m} = T$.
\label{thm-blanczos}
\end{theorem}

%Lemma~\ref{lemma-blanczos} and
%Theorem~\ref{thm-blanczos} can also be generalized to the case $d < n$, where $d$ is divisible by $p$, similarly as for the Lanczos algorithm; see Section~\ref{sec:flan}. 

\subsection{Linear solvers}

In the previous we discussed algorithms for computing bases of the 
corresponding subspaces. We have shown that if the input data have the prescribed nonzero structure, then the basis (block) vectors as well as the projected matrices (defined through the coefficients that appear in the algorithms) are computed exactly. 

The general idea of linear solvers is to look for an approximate solution $x_k$ as a linear combination of the basis vectors. The coefficients 
of the linear combination are defined to be the solution of the 
projected problem. If the algorithm 
for computing the basis is exact, then the projected problem 
is given exactly. To obtain $x_k$, we have 
to solve the (exact) projected problem numerically. Hence, $\bar{x}_k$ is influenced only by 
rounding errors arising when solving the projected problem. Note that 
projected problems are solved using direct methods like
Cholesky or QR factorizations, whose numerical behavior is well understood; see, e.g., \cite{B:Hi2002}. 
In summary, one can expect that 
the computed approximate solution $\bar{x}_k$ is close to $x_k$, if the projected problem is solved accurately.

To demonstrate the above general ideas, consider for example the Arnoldi algorithm, see Section~\ref{subsec:Arnoldi}, applied to $A$ and $v$ having the structure described by Theorem~\ref{thm-arnoldi}. 
For simplicity assume that $\|v\|=1$.
Then $V_k$ as well as 
$$
	H_{k+1,k} \equiv 
	\left[\begin{array}{c}
H_{k} \\
h_{k+1,k}e^T \\
\end{array}\right]
$$
are computed exactly. Starting with $x_0=0$, the GMRES method \cite{SaSch1986} constructs
approximations $x_k$ to the solution of $Ax=v$ of the form 
$$
	x_k = V_k y_k,\qquad y_k =
	\arg\min_y \left\| H_{k+1,k} y -e_1 \right\|, 
$$
where the least squares problem is solved numerically using the QR factorization. Denote the computed coordinate vector by $\bar{y}_k$. Then 
the computed approximate solution $\bar{x}_k$ satisfies 
$$
	\| x_k - \bar{x}_k \| = 	\| V_k y_k - V_k \bar{y}_k \| = 
	\| y_k - \bar{y}_k \|.
$$
Similar consideration can be made for other linear solvers 
that are based on algorithms 
discussed in Sections~\ref{subsec:nlanczos}--\ref{subsec:blockLanc}.

%\cred{Let us suppose we want to solve a system of linear equations. Among other options, we can use a iterative linear solver. In general, linear solvers project in each iteration the original problem onto a smaller space, solve the projected problem and transform the approximation of the solution from the basis of the smaller space to the standard basis. The projection of the original problem is implemented via some algorithm, which computes an orthonormal basis of the smaller space. As it was mentioned before, all the algorithms presented in this section generate bases of some specified spaces and each of the algorithms can be connected to a linear solver. In particular, the Golub-Kahan bidiagonalization is used in LSQR, first described in \cite{PaSa1982}, the block Lanczos algorithm is part of the block variant of CG, defined in \cite{Ol1980}, and the Arnoldi algorithm is core of GMRES, \cite{SaSch1986}. Based on the results presented in this section, we know that for a category of inputs defined previously in this section, the algorithms providing the bases compute exactly. The transformation of the approximation of the solution between the bases therefore does not produce any rounding error. This implies that the inaccuracy of the approximation of the solution is in finite precision dependent only on errors made while solving the projected equation. Since the linear solvers discussed above solve this equation in numerically stable way, we can expect that these linear solvers provide the solution "almost exactly" for the mentioned inputs. }
%
\section{Context and application of results}
\label{sec:cons}

In this section we discuss the application of our results related
to the Lanczos algorithm (Sections~\ref{sec:flan} and \ref{sec:cg}). Analogous considerations can be made also
for other methods discussed in Section~\ref{sec:other}.

\subsection{Various representatives of the original problem\label{subsec:distribution}}
Let $M,\,N  \in \mathbb{R}^{n\times n}$ be symmetric matrices, and let $r,\,s\in \mathbb{R}^{n}$.
We define an {\em equivalence relation} 
in the following way. 
We say that the problem represented by $(M,r)$ is equivalent to the problem 
$(N,s)$, if there is an orthogonal matrix $Q\in \mathbb{R}^{n\times n}$ such that $N = Q^TMQ$ and 
$s = Q^T r$. Having defined the equivalence relation, one may split the set of all couples $(M,c)$  into {\em equivalence classes}.

In the case of the Lanczos algorithm, the original problem is represented by a symmetric matrix $A$ and
a unit norm starting vector~$v_{1}$, 
so that  all equivalent problems are of the form $(Q^{T}AQ,Q^T v_1$).
The equivalence of problems can also be seen via the distribution
function $\omega(\lambda)$ that corresponds to the original data.
Let $U\Lambda U^{T}$ be
the spectral decomposition of $A$, where $U=\left[u_{1},\dots,u_{n}\right]$
is orthogonal and $\Lambda=\mathrm{diag}(\lambda_{1},\dots,\lambda_{n})$.
Assume for simplicity that the eigenvalues $\lambda_{i}$ of $A$
are distinct and increasingly ordered.
For $i=1,\dots,n$ denote
\[
\omega_{i}\equiv\left(v_{1}^{T}u_{i}\right)^{2}\quad\mbox{so that}\quad\sum_{i=1}^{n}\omega_{i}=1.
\]
The distribution function $\omega(\lambda)$ that corresponds to $A$
and $v_{1}$ is defined using
\begin{equation}
\omega(\lambda)\equiv\;\left\{ \;\begin{array}{rcl}
0 & \textnormal{for} & \lambda<\lambda_{1}\,,\\[1mm]
\sum_{j=1}^{i}\omega_{j} & \textnormal{for} & \lambda_{i}\leq\lambda<\lambda_{i+1}\,,\quad1\leq i\leq n-1\,,\\[1mm]
1 & \textnormal{for} & \lambda_{n}\leq\lambda\,;
\end{array}\right.\,\label{eq:schema_hermit}
\end{equation}
see Figure~\ref{fig:df}.
\begin{figure}[ht!]
\begin{center} \unitlength 0.8mm 
\begin{picture}(150,68) \linethickness{0.1pt} 
\put(7,10){\line(1,0){143}}\linethickness{0.1pt} 
\put(8,9){\line(0,1){60}} 
\multiput(7,68)(4,0){34}{\line(1,0){2}} 
\put(10,10){\line(1,0){10}}\linethickness{1.6pt} 
\put(20,10){\line(0,1){8}}\linethickness{0.1pt} 
\put(20,18){\line(1,0){5}}\linethickness{1.6pt} 
\put(25,18){\line(0,1){5}}\linethickness{0.1pt} 
\put(25,23){\line(1,0){10}}\linethickness{1.6pt} 
\put(35,23){\line(0,1){7}}\linethickness{0.1pt} 
\put(35,30){\line(1,0){5}}\linethickness{1.6pt} 
\put(40,30){\line(0,1){10}}\linethickness{0.1pt} 
\put(40,40){\line(1,0){30}}\linethickness{1.6pt} 
\put(70,40){\line(0,1){3}}\linethickness{0.1pt} 
\put(70,43){\line(1,0){25}}\linethickness{1.6pt} 
\put(95,43){\line(0,1){5}}\linethickness{0.1pt} 
\put(95,48){\line(1,0){10}}\linethickness{1.6pt}
\put(105,48){\line(0,1){2}}\linethickness{0.1pt}
\put(105,50){\line(1,0){5}}\linethickness{1.6pt}
\put(110,50){\line(0,1){8}}\linethickness{0.1pt}
\put(110,58){\line(1,0){20}}\linethickness{1.6pt}
\put(130,58){\line(0,1){5}}\linethickness{0.1pt}
\put(130,63){\line(1,0){10}}\linethickness{1.6pt}
\put(140,63){\line(0,1){5}}\linethickness{0.1pt}
\put(140,68){\line(1,0){10}} 

\put(20,11){\line(0,-1){3}}
\put(25,16){\line(0,-1){3}}
\put(25,11){\line(0,-1){3}}
\put(35,21){\line(0,-1){3}}
\put(35,16){\line(0,-1){3}}
\put(35,11){\line(0,-1){3}}
\put(140,60){\line(0,-1){3}}
\put(140,54){\line(0,-1){3}}
\put(140,48){\line(0,-1){3}}
\put(140,29){\line(0,-1){3}}
\put(140,23){\line(0,-1){3}}
\put(140,17){\line(0,-1){3}}
\put(140,11){\line(0,-1){3}}

\put(140,38){\makebox(0,0)[cc]{$\vdots$}} 
\put(2,10){\makebox(0,0)[lc]{0}} 
\put(2,68){\makebox(0,0)[cc]{1}}
\put(14,12){\makebox(0,0)[lb]{$\omega_1$}}
\put(24,19){\makebox(0,0)[rb]{$\omega_2$}}
\put(34,24.5){\makebox(0,0)[rb]{$\omega_3$}}
\put(39,33){\makebox(0,0)[rb]{$\omega_4$}}
\put(139,64){\makebox(0,0)[rb]{$\omega_{n}$}}  
\put(20,5){\makebox(0,0)[ct]{$\lambda_1$}}
\put(26,5){\makebox(0,0)[ct]{$\lambda_2$}}
\put(36,5){\makebox(0,0)[ct]{$\lambda_3$}}
\put(57,4){\makebox(0,0)[ct]{\dots}}
\put(120,4){\makebox(0,0)[ct]{\dots}}
\put(141,5){\makebox(0,0)[ct]{$\lambda_{n}$}}
\end{picture}
\end{center}
\label{fig:df}
\caption{The distribution function $\omega(\lambda)$.}
\end{figure}
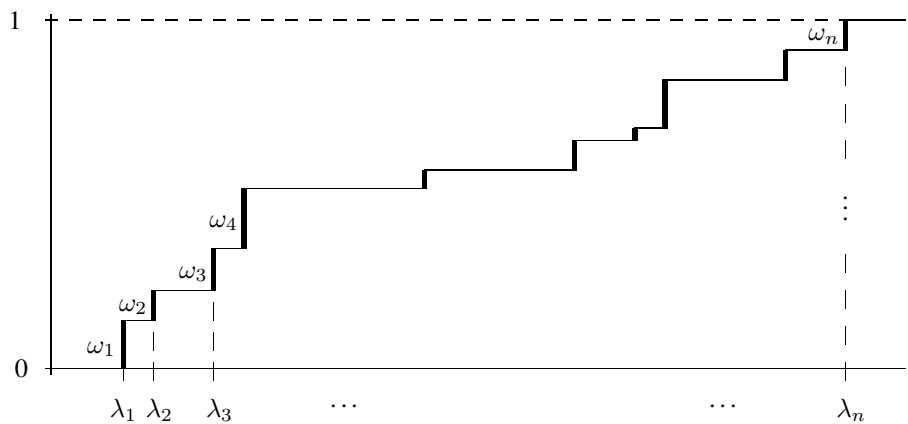

If two problems share the same distribution function, then there exists
an orthogonal matrix $Q$ that transforms one problem into the other, i.e.,
the problems are equivalent.
All problems with the same distribution function form an equivalence class,
and $(A,v_{1})$ can be seen as a representative of this equivalence class.
 Another representative is $(\Lambda,w)$, where
\[
w\equiv\begin{bmatrix}\omega_{1}^{1/2}, & \dots, & \omega_{n}^{1/2}\end{bmatrix}^{T},
\]
or $(\Lambda,\tilde{w})$, where $\tilde{w}\equiv U^{T}v_{1}$. Finally, assuming
for simplicity that $\omega_{i}\neq0$ for $i=1,\dots,n$, it holds
that $d=n$, and $(T_{n},e_1)$
resulting from the exact Lanczos
algorithm applied to $A$ and $v_1$ stands for yet another representative; see Figure~\ref{fig:repre}. 
Therefore, any theoretical behaviour of the Lanczos algorithm (represented by the generated tridiagonal matrices $T_k$)  can be observed for 
the initial data having 
the structure
\eqref{eq:structure}. In other words, 
concentrating on test problems having 
the structure \eqref{eq:structure}
is not restrictive and covers any theoretical behaviour of the Lanczos algorithm.
\begin{figure}[ht]
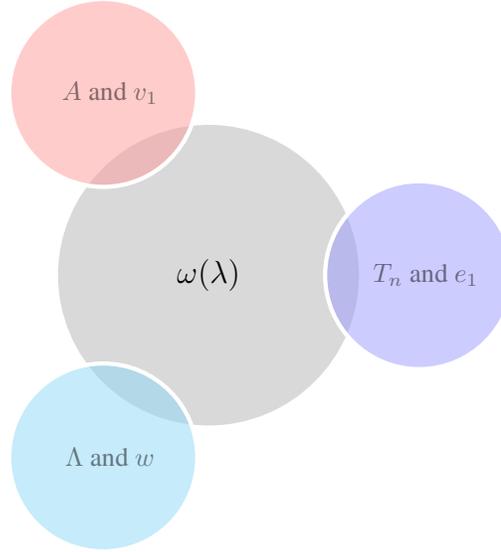

\begin{center} 
\smartdiagram[bubble diagram]{
$\omega(\lambda)$, $A$ and $v_{1}$, $\Lambda$ and $w$, $T_n$ and $e_{1}$}   
\end{center}
\caption{Various representatives of $\omega(\lambda)$.}
\label{fig:repre}
\end{figure}

The representative $(\Lambda,w)$ provides directly the key information
about the distribution function. On the other hand, $(T_{n},e_{1})$
is the only representative that guarantees that the Lanczos algorithm
(or the corresponding Stieltjes process, see, e.g., \cite{GrHa1984,OLStTi2007})
%applied to this data 
will not be affected by rounding errors; see Theorem~\ref{thm:main}. 

Assuming $d=n$ and having one of the representatives, one can ask
how to compute the other representatives in a {\em numerically reliable
way}. Starting from $(A,v_{1})$, we can find $(T_{n},e_{1})$
using the Lanczos (or Arnoldi) algorithm with double reorthogonalization
\cite{GrSt1992,GiLaRoEs2005}. If the double reorthogonalization is
not used, the rounding errors can strongly influence the computations,
and the computed $\bar{T}_{n}$ can be completely different from the exact
$T_{n}$. Instead of double reorthogonalization, one can alternatively
use Householder reflections to transform $(A,v_{1})$ to $(\tilde{T},\tilde{v}_{1}),$
where $\tilde{T}$ is tridiagonal, and then Givens rotations to transform
$\tilde{v}_{1}$ to $e_{1}$ while preserving the tridiagonal
structure of the transformed matrix using the chasing the bulge strategy.
In general, to compute the representative $(T_{n},e_{1})$ in a numerically
reliable way, one has to store a dense matrix, and
the cost of computations is then $\mathcal{O}(n^{3})$ flops.

Concerning the other two representatives, there exist numerically reliable transformations between $(T_{n},e_{1})$
and $(\Lambda,w)$ with the cost of $\mathcal{O}(n^{2})$ flops and
low memory requirements. In more detail, starting from $(T_{n},e_{1})$,
one can use the Golub-Welsh algorithm~\cite{GoWe1969} to compute $(\Lambda,w)$.
In the opposite way, having $(\Lambda,w)$, the \texttt{rkpw} algorithm
of Gragg and Harrod \cite{GrHa1984} or the \texttt{pftoqd} algorithm
of Laurie \cite{La1999} are capable to compute $(T_{n},e_{1})$ reliably.

\subsection{Any theoretical behavior is observable also numerically\label{subsec:theoretical}}

To investigate theoretical as well as numerical behaviour of Krylov
subspace methods, it is crucial to ask convenient questions that help in understanding complicated phenomenons. Here we discuss three
questions of that kind.

An important question asked in literature, see, e.g., \cite{HeSt1952,Sc1979,GrSt1994,GrPtSt1996,TeMe2012,Me2020},
is about {\em possible theoretical behaviour} of the considered method.
For example, in the case of the conjugate gradient method, one can
prescribe any decreasing convergence curve for the $A$-norm of the
error and, at the same time, any convergence curve for the residual
norms (positive numbers), and then construct a symmetric positive definite matrix~$A$
and a right hand side $b$ such that exact CG applied to $Ax=b$ generates
the prescribed convergence curves; see \cite{HeSt1952,Me2020}.
In more detail, the CG coefficients $\delta_{k}$, see Algorithm~\ref{alg:cg},
satisfy
\[
\delta_{k}=\frac{\|r_{k}\|^{2}}{\|r_{k-1}\|^{2}}.
\]
Therefore, if the residual norms are given, then $\delta_{k}$'s are
known. Moreover, since 
\[
\|x-x_{k}\|_{A}^{2}=\gamma_{k}\|r_{k}\|^{2}+\|x-x_{k+1}\|_{A}^{2},
\]
see \cite{HeSt1952}, and residual norms as well $A$-norms
of the error are prescribed, also $\gamma_{k}$'s are known. Finally,
as discussed in Section~\ref{sec:cg}, CG computes implicitly the $LDL^{T}$ factorization
of the tridiagonal matrix $T_{k}$. Assuming again for simplicity
that $d=n$, the coefficients $\delta_{1},\dots,\delta_{n-1}$ and
$\gamma_{0},\dots,\gamma_{n-1}$ determine uniquely the tridiagonal
matrix $T_{n}$. Defining $A=T_{n}$ and $b=\|r_{0}\|e_{1}$, we obtain
a system of linear equations such that exact CG applied to $Ax=b$
generates the prescribed residual norms and $A$-norms of the error.
For more details and the related discussion, see \cite{Me2020}. Let
us emphasize that the constructed matrix is a Jacobi matrix and that
the right hand side vector is a multiple of $e_{1}$. 

Another question that can help in understanding numerical behaviour
of Krylov subspace methods is the following one. Can the observed
{\em numerical behaviour be interpreted as the behaviour of the exact algorithm}
applied to a problem that is, in some sense, close to the original
one? In other words, we would like to find a mathematical model of
the results of finite precision computations of the considered algorithm.
Note that the term ``a problem close to the original one'' can have
different meanings. For example, it can be understood in the classical
backward error sense, i.e., one can look for a small perturbation
of the original data, or, as in the case of the Lanczos and CG algorithms,
one can look for a small perturbation of the distribution function
$\omega(\lambda)$ discussed in Subsection~\ref{subsec:distribution}.
In particular, Greenbaum \cite{Gr1989} showed that the results of
the finite precision Lanczos algorithm can be interpreted as the results
of the exact Lanczos algorithm applied to a larger problem with a
matrix having clustered eigenvalues around the original eigenvalues
of $A$. The perturbed distribution function has larger support (clusters
of eigenvalues) and the sum of weights that correspond to the $i$th
cluster is equal to the original weight $\omega_{i}$. Note that the
larger matrix, used in \cite{Gr1989} for simulating the behaviour
of the finite precision Lanczos algorithm, was a Jacobi matrix, and
that the starting vector was a multiple of $e_{1}$.
Analogous results can be obtained also for CG, but here the exact CG algorithm applied to the model
problem will not generate exactly the same convergence curves (the residual norms and the $A$-norms of the error) as the finite precision CG algorithm applied to the original data. However, it will generate their very close approximations; see \cite{Gr1989}. 

We now comment on our results. We have shown that if the matrix $A$ and the starting vector $v$
have the structure described in \eqref{eq:structure}, then the Lanczos algorithm computes exactly
in the standard floating point arithmetic. Moreover, a variant of CG
(Algorithm~\ref{alg:lanczos-cg}) applied to a $Ax=b$ where $A$ and $b$ have the
structure \eqref{eq:structure}, computes almost exactly. Hence, since
the above mentioned questions lead to systems having the structure 
\eqref{eq:structure}, our results allow to check
the answers numerically and without reorthogonalization, even for large problems.

Moreover, results of this paper give the answer to the following question:
Can any {\em theoretical behavior} of the Lanczos and CG algorithms {\em be observed
also numerically} (up to the relative accuracy limited by machine precision), without using reorthogonalization or extended precision arithmetic?
In more detail, the theoretical behaviour of the Lanczos algorithm is represented by the generated matrices $T_k$. As discussed in Section~\ref{subsec:distribution}, any theoretical behavior 
of the Lanczos algorithm can be observed for the representative $(T,e_1)$, where $T\in\mathbb{R}^{n\times n}$ is a Jacobi matrix. 
The results of the exact Lanczos algorithm applied to $T$ and $e_1$ are then represented 
by the leading principal submatrices of $T$.
Converting the matrix $T$ into the considered floating
point arithmetic we obtain $\bar{T}= \mathrm{fl}(T)$, and the data $(\bar{T},e_1)$ have the structure \eqref{eq:structure}. Therefore, 
the finite precision Lanczos algorithm applied to 
$\bar{T}$ and $e_1$ computes exactly, i.e., it generates
the leading principal submatrices $\bar{T}_k$ of $\bar{T}$, 
and it holds that $\bar T_k = \mathrm{fl}(T_k)$.
%; see \eqref{eqn:fl1}.
In this sense, any theoretical behavior of the Lanczos algorithm represented by real matrices $T_k$ can be observed also numerically.
Using the results of Section~\ref{sec:cg}, analogous conclusion holds also for CG implemented using Algorithm~\ref{alg:lanczos-cg}.

\subsection{Theoretical study of (large) model problems}

Numerical experiments studying {\em theoretical behaviour} of the CG and Lanczos algorithms, but also of other methods and algorithms mentioned in 
Section~\ref{sec:other},
are in general restricted to relatively small problems only. To be sure that the computed results agree with the exact results one  either has to reorthogonalize or to use extended precision arithmetic.
To reorthogonalize, all the basis vectors have to be stored, and memory requirements do not allow to handle large problems. When using extended precision arithmetic, one usually has to consider a huge number of valid digits leading to very slow computations even for small problems.

The results of this paper provide
a new practical tool for the analysis of the theoretical as well as finite precision behaviour of the considered algorithms, including the analysis of the behaviour of error estimates, e.g.,  the error estimates of the $A$-norm of the error in CG. We can use this tool
in the standard floating point arithmetic,  without reorthogonalization, and for potentially very large problems.
We only have to be able to construct model problems having 
the desired properties
and the prescribed nonzero structure. Then, the prescribed nonzero structure ensures that the finite precision computations are exact for the algorithms that compute the basis vectors, and almost exact for the corresponding linear solvers.

%The theory presented in Sections~\ref{sec:flan}--\ref{sec:other} give us a
%new practical tool to study theoretical as well as finite precision behaviour of the considered algorithms for potentially very large problems. We only have to construct model problems having the prescribed nonzero structure, and then the finite precision computations are exact for the algorithms computing the basis vectors, and almost exact for the corresponding linear solvers.
%For example, we can prescribe 
%the eigenvalues and the weights, and construct the corresponding tridiagonal matrix $T$ using the \texttt{rkpw}
%or \texttt{pftoqd} algorithms. Then, we can study the theoretical
%behaviour of the Lanczos and CG algorithms applied to $T$ and $e_{1}$;
%see also \cite{OLStTi2007}. 

\bibliographystyle{plain}

\end{document}